\documentclass[10pt]{article}

\usepackage{amsmath} %pacchetto utilie per il linguaggio matematico
\usepackage{amsthm} %pacchetto utilie per il linguaggio matematico
\usepackage{amssymb} %pacchetto utilie per il linguaggio matematico
\usepackage{graphicx} %pacchetto utilie per stampare le immaggini
\usepackage{nicefrac} %pacchetto utilie per stampare le immaggini
\usepackage{longtable} %pacchetto utilie per spezzare le tabelle
\usepackage{textcomp} %pacchetto utilie per spezzare le tabelle

\usepackage{paralist}

\usepackage{hyperref}%to link to webpage
\usepackage[small,nohug,heads=vee]{diagrams}
\diagramstyle[labelstyle=\scriptstyle]

\usepackage[table]{xcolor}
\usepackage{xy}
\usepackage{tikz}

\definecolor{cof}{RGB}{219,144,71}
\definecolor{pur}{RGB}{186,146,162}
\definecolor{greeo}{RGB}{91,173,69}
\definecolor{greet}{RGB}{52,111,72}

\usepackage{array}
\newcolumntype{L}[1]{>{\raggedright\let\newline\\\arraybackslash\hspace{0pt}}m{#1}}
\newcolumntype{C}[1]{>{\centering\let\newline\\\arraybackslash\hspace{0pt}}m{#1}}
\newcolumntype{R}[1]{>{\raggedleft\let\newline\\\arraybackslash\hspace{0pt}}m{#1}}

\newtheorem{teo}{Theorem}[section]
\newtheorem{lem}[teo]{Lemma}
\newtheorem{pro}[teo]{Proposition}
\newtheorem{cor}[teo]{Corollary}

\newtheorem{defi}[teo]{Definition}
\newtheorem{Obs}[teo]{Observation}

\newtheorem*{notation}{Notation}

\newtheorem*{conj}{Conjecture}
\newtheorem*{Question}{Question}%per non contare le definizioni di torsore e BG nell'introduzione.

\newcommand{\rat}{\mathbb{Q}}
\newcommand{\nat}{\mathbb{N}}
\newcommand{\com}{\mathbb{C}}
\newcommand{\inte}{\mathbb{Z}}

\newcommand{\Ffield}{\mathbb{F}}
\newcommand{\Field}{\mathbb{K}}
\newcommand{\field}{\mathbf{k}}

\newcommand{\Zp}{\nicefrac{\inte}{p\inte}}

\newcommand{\Gl}[2]{\operatorname{GL}_{#1}\left(#2\right)}

\newcommand{\Ab}{\mathbf{Ab}}

\newcommand{\Var}[1]{\mathbf{Var_{#1}}}
\newcommand{\Stck}[1]{\mathbf{Stack_{#1}}}

\newcommand{\hei}[1]{H_{#1}}
\newcommand{\Hom}[3]{\operatorname{H}^{#1}\left(#2; #3\right)}

\newcommand{\Homred}[2]{\operatorname{H}^{#1}\left(#2\right)}

\newcommand{\grot}[1]{K_0(#1)}
\newcommand{\grotcom}[1]{\widehat{K_0}(#1)}

\newcommand{\Lo}[1]{L_0(#1)}

\newcommand{\Lclass}{\mathbb{L}}

\newcommand{\eke}[2]{\operatorname{e}_{#1}\left(#2\right)}

\newcommand{\B}[1]{\operatorname{\mathcal{B}} #1}
\newcommand{\cl}{\{\B{G}\}}

\title{Introduction to the Ekedahl Invariants}
\author{Ivan Martino}

\begin{document}
\maketitle

\begin{abstract}
In 2009 T. Ekedahl introduced certain cohomological invariants for finite groups. 
In this work we present these invariants and we give an equivalent definition that does not involve the notion of algebraic stacks.
Moreover we show certain properties for the class of the classifying stack of a finite group in the Kontsevich value ring.
\end{abstract}

%In this work we introduce the Ekedahl Invariants for a finite group.
%
In \cite{EkedahlStack}, Ekedahl studied whether the class of the classifying stack $\cl$ of a group $G$ equals the class of a point $\{*\}$ in the Grothendieck group of algebraic stacks.
All the known examples of finite groups when this does not happen are the counterexamples to the Noether problem:
Let $\Ffield$ be a field and consider the extension $\Ffield \subset \Ffield(x_g: g\in G)^G$, one wonders if this is rational (see \cite{Noether1917}), i.e. purely transcendental.
To show in which cases $\cl\neq \{*\}$, Ekedahl introduced in \cite{Ekedahl-inv} a new kind of geometric invariants for finite groups defined as the \emph{cohomology} for the class of classifying stack of $G$ in Kontsevich value ring of algebraic varieties.

Let $\field$ be an algebraically closed field of characteristic zero.
We denote by $\grot{\Var{\field}}$ the Grothendieck group of algebraic $\field$-varieties. 
Let $\mathbb{L}^i$ be the class of the affine space $\mathbb{A}^i_{\field}$ in $\grot{\Var{\field}}$ (so $\mathbb{L}^0=\{*\}$, the class of a point).
Let $\grotcom{\Var{\field}}$ be the Kontsevich value ring of algebraic $\field$-varieties (see Section \ref{sec-Preliminaries}).  
We note that in this ring $\mathbb{L}^i$ is invertible.

We denote by $\Lo{\Ab}$ the group generated by the isomorphism classes $\{G\}$ of finitely generated abelian groups $G$ under the relation $\{A\oplus B\}=\{A\}+\{B\}$.

\noindent
For every integer $k$, in \cite{EkedahlStack} Ekedahl defines a \textit{cohomological} map
\[
  \operatorname{H}^k: \grotcom{\Var{\field}}\rightarrow \Lo{\Ab}
\]
by assigning $\operatorname{H}^k(\nicefrac{\{X\}}{\mathbb{L}^m})=\{\Hom{k+2m}{X}{\mathbb{Z}}\}$ for every smooth and proper $\field$-variety $X$ (see Section \ref{sec-ekedahl-invariants}).
% 
% \noindent

Let $V$ be a $n$-dimensional faithful $\field$-representation of a finite group $G$ and let $G$ act component-wise on the product $V^m$. 
Thus consider the quotient $\nicefrac{V^m}{G}$.
In Proposition \ref{prop-property-BG} we show that $\cl$ is equal to $\lim_{m\rightarrow \infty} \{\nicefrac{V^{m}}{G}\} \Lclass^{-mn}$ in $\grotcom{\Var{\field}}$.
So one defines:
\setcounter{section}{3}\setcounter{teo}{1}
\begin{defi}
  For every integer $i$, the $i$-th \emph{Ekedahl invariant} $\eke{i}{G}$ of the group $G$ is $\operatorname{H}^{-i}(\{B{G}\})$ in $\Lo{\Ab}$.
  We say that the Ekedahl invariants of $G$ are \emph{trivial} if $\eke{i}{G}=0$ for $i\neq 0$.
\end{defi}

The purpose of this paper is to introduce the theory of the Ekedahl invariants to a reader who is not used to the notion of the algebraic stacks.
For this reason we also present some non-published results from \cite{Ekedahl-inv, EkedahlStack} aiming to a complete and self contained survey of the topic.
The author believes that one could work with Ekedahl invariants with basic knowledge of algebraic geometry and for this reason we present the following \emph{non-stacky} definition:

\setcounter{section}{4}\setcounter{teo}{1}
\begin{defi}
  Let $V$ be a $n$-dimensional faithful $\field$-representation of a finite group $G$ and let $X$ be a smooth and proper resolution of $\nicefrac{V^m}{G}$:
  \[
    X\xrightarrow{\,\,\,\,\,\,\,\pi\,\,\,\,\,\,\,} \nicefrac{V^m}{G}.
  \]
  
  \noindent
  There exists a positive integer $m=m(i, V)$, depending only by $i$ and $V$, so that the $i$-th Ekedahl invariant is defined as follows:
  \[
    \eke{i}{G}=\{\Hom{2m-i}{X}{\inte}\}+\sum_j n_j\{\Hom{2m-i}{X_j}{\inte}\} \in \Lo{\Ab},
  \]
  where $\{\nicefrac{V^m}{G}\}\in \grot{\Var{\field}}$ is written as the sum of classes of smooth and proper varieties $\{X\}$ and $\{X_j\}$, $\{\nicefrac{V^m}{G}\}=\{X\}+\sum_j n_j\{X_j\}$. 
\end{defi}

\noindent
In Proposition \ref{pro-coh-stabilizes}, we prove that such $m(i,V)$ always exists and, then, in Proposition \ref{pro-defi-equivalent} we show that the latter definition is well given and does not depend on the choice of the representation $V$ and of the resolution $X$.
Finally we prove that the two given definitions are equivalent.

We observe that the new definition does not involve the notion of algebraic stacks and for this reason could be appreciated by a more wider audience. 

We use the latter to prove the following theorem due to Ekedahl (see Theorem 5.1 in \cite{Ekedahl-inv}):
\setcounter{section}{4}\setcounter{teo}{3}
\begin{teo}
%   Let $G$ be a finite group. 
%   %
%   Let $\field$ be a field of characteristic zero.
  %
  We denote by $B_0(G)^{\vee}$ the pontryagin dual of the Bogomolov multiplier of the group $G$.
  If $G$ is a finite group, %and if $\field$ is an algebraically closed field of characteristic zero, 
  then
  \begin{description}
    \item[a)] $\eke{i}{G}=0$, for $i<0$;
    \item[b)] $\eke{0}{G}=\{\mathbb{Z}\}$;
    \item[c)] $\eke{1}{G}=0$;
    \item[d)] $\eke{2}{G}=\{B_0(G)^{\vee}\}+\alpha\{\inte\}$ for some integer $\alpha$.
%     \item[d)] $\eke{2}{G}=\{B_0(G)^{\vee}\}$, where $B_0(G)^{\vee}$ is the dual of the Bogomolov multiplier of the group $G$;
%     \item[e)] For $i>0$, $\eke{i}{G}$ is the sum (with signs) of classes of finite groups in $\Lo{\Ab}$.
  \end{description}
\end{teo}

Item \textbf{d)} is related to the Noether problem (see Section \ref{sec-Preliminaries}).
Ekedahl actually proved a stronger version of this result since he showed that $\eke{2}{G}=\{B_0(G)^{\vee}\}$.

In Section \ref{sec-Preliminaries}, after a brief historical introduction, we set all the basic definitions and notations.
In Section \ref{sec-classifying-stack} we discuss some properties of the class of the classifying stack $\cl$ and in the next section \ref{sec-ekedahl-invariants} we define the Ekedahl invariants as the \emph{cohomology} of the classifying stack.
In Section \ref{sec-nonstaky} we present the equivalent \emph{non-stacky} definition.
Finally, in Section \ref{sec-literature}, we use this definition to reprove partially Theorem 5.1 of \cite{Ekedahl-inv}.
In the end of the article we recall the state of the art of the Ekedahl invariants.

%Even if it looks complicate in all the studied cases this becomes easy to deal with (see for instance \cite{Martino-TheEkedahl}, ...)

\begin{notation}\em
In all this manuscript $\Ffield$ is a field and $\field$ is an algebraically closed field of characteristic zero. We set $*=\operatorname{Spec}(\field)$.
If $X$ is a scheme with a $G$-action, then we denote by $\nicefrac{X}{G}$ the schematic quotient and by $[\nicefrac{X}{G}]$ the stack quotient. 

% In all this manuscript, $G$ is a finite group and we work over an algebraically closed field $\field$ of characteristic zero.
% In Section \ref{sec-Preliminaries} and \ref{sec-classifying-stack} 
% we work over an algebraically closed field $\field$ of characteristic zero.
%
%
% Moreover every cohomology group (if not explicitly expressed differently) is the singular cohomology group with integer coefficient, that is $\Homred{k}{-}=\Hom{k}{-}{\inte}$.

In Sections \ref{sec-Preliminaries} and \ref{sec-classifying-stack}, $G$ is a linear group without any assumption of finiteness. In Sections \ref{sec-ekedahl-invariants} and \ref{sec-nonstaky}, instead, $G$ is always a finite group. 
\end{notation}

\setcounter{section}{0}\setcounter{teo}{0}
\section{Preliminaries}\label{sec-Preliminaries}
Let $\Ffield$ be a field and let $G$ be a finite group.
We denote by $\Ffield(x_g: g\in G)$ the field of rational functions with variables indexed by the elements of the group $G$.
The group acts on it via $h\cdot x_g=x_{hg}$.
We consider the field extension 
\begin{equation}\label{eq-extension}
  \Ffield \subset \Ffield(x_g: g\in G)^G, 
\end{equation}
where the latter denotes the field of invariants.
In 1914, Emmy Noether in \cite{Noether1917} wondered if the field extension (\ref{eq-extension}) is rational (i.e. purely transcendental).

\noindent
Mathematicians conjectured a positive answer to the Noether problem, until the breakthrough result of of Swan \cite{Swan1969} in 1969. 
He proved that the extension $\rat \subset \rat(x_g: g\in \Zp)^{\Zp}$ is not rational for $p=47, 113$ and $233$.

\noindent
After this, Saltman \cite{Saltman1984} proved that for any prime $p$ with $(\operatorname{char}\Ffield, p)=1$, there exists a group $G$ of order $p^9$ such that the Noether problem has negative answer.
He used a cohomological invariant introduced by Artin and Mumford in \cite{ArtinMumford1972}.
Bogomolov in \cite{Bogomolov1988}, showed a concrete way to compute this invariant that is now called Bogomolov multiplier $B_0(G)$:
\[
        B_0(G)=\bigcap_{A} \operatorname{Ker}\left(\Hom{2}{G}{\mathbb{C}^*}) \rightarrow \Hom{2}{A}{\mathbb{C}^*}\right),
\]
where the intersection runs over the abelian subgroups $A\subseteq G$.
This is a \emph{cohomological obstruction} to the rationality of (\ref{eq-extension}), i.e. the rationality of (\ref{eq-extension}) implies $B_0(G)=0$. 
Bogomolov also improved Saltman's statement from $p^9$ to $p^6$. 
%
%Other results in this direction are \cite{Chu-Hu-Kang-Kunyavskii2010, Chu-Kang2001, Moravec-Brauer}.

\noindent
Recently in Hoshi, Kang and Kunyavskii in \cite{HKM-noether}, classified all the $p^5$-groups with nontrivial Bogomolov multiplies.
%
%They show that $B_0(G)\neq 0$ if and only if $G$ belongs to the isoclinism family $\phi_{10}$. 

In the rest of this section we introduce some definitions which we need later.
%
%Let $\field$ be an algebraic closed field of characteristic zero.
%
The Grothendieck group $\grot{\Var{\field}}$ of varieties over $\field$ is the group generated by the isomorphism classes $\{X\}$ of algebraic $\field$-varieties $X$, subjected to the relation
\[
        \{X\}=\{Z\}+\{X\setminus Z\},
\]
for all closed subvarieties $Z$ of $X$.
It is possible to see that $\grot{\Var{\field}}$ has a ring structure given by $\{X\}\cdot \{Y\}=\{X\times Y\}$.

\noindent
The class of the empty set $\{\emptyset\}$ is also denoted with $0$, the class of the point $\{*\}$ with $1$ and the class of the affine line $\mathbb{L}=\{\mathbb{A}^1_{\field}\}$ is called Lefschetz class.
Using the multiplication operation, one gets $\{\mathbb{A}^n_{\field}\}=\mathbb{L}^n$ and $\{\mathbb{P}^n_{\field}\}=\mathbb{L}^0+\mathbb{L}^1+\dots+\mathbb{L}^n$. 

%Let $\field$ be field of characteristic zero. 
%
In \cite{Bittner2004}, Bittner proves that $\grot{\Var{\field}}$ is generated by the class of smooth and proper varieties modulo the relations 
$\{X\}+\{E\}=\{Bl_{Y}(X)\}+\{Y\}$
%\[
%	\{X\}+\{E\}=\{Bl_{Y}(X)\}+\{Y\}
%\]
with $Bl_{Y}(X)$ being the blow up of $X$ along $Y$ with exceptional divisor $E$:
\begin{diagram}
Bl_{Y}(X)       &       \rTo    & X\\
\uInto            &               & \uInto\\
E               &       \rTo    & Y.
\end{diagram}
Therefore, with the help of compactification and resolution of singularities one writes the class of a scheme $\{X\}\in \grot{\Var{\field}}$ as a sum of classes of smooth and proper varieties $\{X_j\}$: $\{X\}=\sum_j n_j\{X_j\}$, with $n_j\in \inte$.

The Motivic ring of algebraic $\field$-varieties is $\grot{\Var{\field}}[\mathbb{L}^{-1}]$.
%
%SOME EXAMPLE HERE.
%
We naturally define a filtration in dimension
\[
        Fil^{n}\left(\grot{\Var{\field}}[\mathbb{L}^{-1}]\right)=\{\nicefrac{\{X\}}{\mathbb{L}^i}: \dim{X}-i\leq n\}.
\]
We denote by $\grotcom{\Var{\field}}$ the completion of the Motivic ring with respect to this filtration. This ring is called Kontsevich's value ring.

\begin{defi}
  Let $X$ be a scheme over $\field$. 
  A $G$-torsor $P$ over $X$, $P\rightarrow X$, is a scheme with a free $G$-action that is locally trivial over $S$.
\end{defi}

% We recall that a $G$-torsor $P$ over a scheme $X$, $P\rightarrow X$, is a $G$-scheme, locally of the form $U \times_S G$, for a suitable open set $U$ of $X$.
%
\begin{defi}%\label{def-classifying-stack}
  The classifying stack $\B{G}$ of a group $G$ is a pseudo-functor from the category of schemes over $\field$, $\mathbf{Sch}_{\field}$, to the category of groupoids over $\field$, $\mathbf{Gpd}_{\field}$, sending any open scheme $U$ to the groupoid of $G$-torsors over $U$:
  \begin{eqnarray*}
	\B{G}: \mathbf{Sch}_{\field}	&\rightarrow& 	\mathbf{Gpd}_{\field}\\
		U			&\mapsto &	\{G\text{-torsors over }U\}.
  \end{eqnarray*}
\end{defi}

\noindent
Equivalently, with the only possible trivial action of $G$ on $*$, the classifying stack of the group $G$ is usually defined as the stack quotient $\B{G}=[\nicefrac{*}{G}]$.

%===========================================================
%===========================================================
\section{The class of the classifying stack}\label{sec-classifying-stack}
To go further, it is necessary to introduce another Grothendieck group for algebraic stacks.

%Recall that we denote by $\grot{\Var{\field}}$, $\grot{\Var{\field}}[\mathbb{L}^{-1}]$ and $\grotcom{\Var{\field}}$ the Grothendieck ring, the Motivic ring and, respectively, the Kontsevich value ring of algebraic $\field$-varieties, where $\field$ is an algebraic closed field of characteristic zero.

\begin{defi}\label{defi-grot-stack}
  We denote by $\grot{\Stck{\field}}$ the Grothendieck group of algebraic $\field$-stacks. This is the group generated by the isomorphism classes $\{X\}$ of algebraic $\field$-stacks $X$ of finite type all of whose automorphism group scheme are affine (shortly, algebraic $\field$-stack of finite type with affine stabilizer). 
  The elements of this group fulfill the following relations:
  \begin{enumerate}
    \item for each closed substack $Y$ of $X$, $\{X\}=\{Y\}+\{Z\}$, where $Z$ is the complement of $Y$ in $X$;
    \item for each vector bundle $E$ of constant rank $n$ over $X$, $\{E\}=\{X\times \mathbb{A}^n\}$.
  \end{enumerate}
\end{defi}

\noindent
Similarly to $\grot{\Var{\field}}$, $\grot{\Stck{\field}}$ has a ring structure.
\begin{lem}\label{lem-grot-localization-and-completion}
  One has that $$\grot{\Stck{\field}}=\grot{\Var{\field}}[\Lclass^{-1}, (\Lclass^{n}-1)^{-1}, \forall n\in \nat].$$
  Moreover, the completion map $\grot{\Var{\field}}[\Lclass^{-1}] \rightarrow \grotcom{\Var{\field}}$ factors through 
  \[
    \grot{\Var{\field}}[\Lclass^{-1}] \rightarrow \grot{\Stck{\field}} %%\stackrel{\phi}{\rightarrow} 
              \rightarrow \grotcom{\Var{\field}}.
  \]
\end{lem}
\begin{proof}
        The first part is proved in Theorem 1.2 of \cite{Ekedahl-inv}. 
        Regarding the second one, we observe that $\Lclass^{n}-1=\Lclass^n(1-\Lclass^{-n})$ is invertible in $\grotcom{\Var{\field}}$.
	Indeed, $(1-\Lclass^{-n})^{-1}=1+\Lclass^{-n}+\Lclass^{-2n}+\dots$ and each truncation $x_k=\sum_{j=0}^k\Lclass^{-kn}$ belongs to $Fil^{-kn}$. So, the serie converges in $\grotcom{\Var{\field}}$.
\end{proof}
%
%We remark that a $G$-torsor $P$ over a scheme $X$, $P\rightarrow X$, is a $G$-scheme, locally of the form $U \times_S G$, for a suitable open set $U$ of $X$.
%
% \begin{defi}
% 	The classifying stack of a group $G$, $\B{G}$, is a pseudo-functor from the category of schemes, $\mathbf{Sch}_{\field}$, to the category of groupoids, $\mathbf{Gpd}_{\field}$, sending any open scheme $U$ to the groupoids of $\{$G$\text{-torsors over }U\}$:
% 	\begin{eqnarray*}
% 		\B{G}: \mathbf{Sch}_{\field}	&\rightarrow& 	\mathbf{Gpd}_{\field}\\
% 			U			&\mapsto &	\{G\text{-torsors over }U\}
% 	\end{eqnarray*}
% \end{defi}
% %
% The classifying stack of the group $G$ is also defined as the stack quotient $[*/G]$.
%

Recall that a special group $W$ is a connected algebraic group scheme of finite type all of whose torsors over any extension field $\field\subseteq \Field$ are trivial.
It is useful for what to note that $GL_n$ is a special group.

Let $W$ be a special group and let $X\rightarrow Y$ be a $W$-torsor of algebraic stacks of finite type over $\field$, then $\{X\}=\{W\}\{Y\}$ in $\grot{\Stck{\field}}$.
Moreover, if $F$ is a $W$-space and if $Z\rightarrow Y$ is the associated fiber space to the $G$-torsor $X\rightarrow Y$, then $\{Z\}=\{F\}\{Y\}$ in $\grot{\Stck{\field}}$.

\noindent
Those two facts are not true for every group $G$. Indeed special groups play an important role in this topic (see Proposition 1.4 of \cite{EkedahlStack}).

\begin{lem}\label{lem-group-basic}
  If $W$ is a special group over $\field$ and if $H$ is a closed subgroup scheme of $W$, then
\begin{description}
  \item [a)] $\{W\}\{\B{W}\}=1$;
  \item [b)] $\{\B{H}\}=\{\nicefrac{W}{H}\}\{\B{W}\}$.
\end{description}
\end{lem}
\begin{proof}
  Consider the $W$-torsor $*\rightarrow [\nicefrac{*}{W}]$. 
  Thus, $\{*\}=\{W\}\{[\nicefrac{*}{W}]\}$.
  Moreover $*\rightarrow [\nicefrac{*}{W}]$ is also an $H$-torsor and the action of $H$ makes $\nicefrac{W}{H}$ a $H$-space.
  There is a natural $\nicefrac{W}{H}$-fibration associated, $[\nicefrac{*}{H}]\rightarrow [\nicefrac{*}{W}]$, given by $\B{H}=\nicefrac{W}{H}\times_{W}\B{W}$ and thus $\{\B{H}\}=\{\nicefrac{W}{H}\}\{\B{W}\}$.
\end{proof}

\noindent
We now consider only finite groups.
\begin{lem}\label{lem-formula-BG}
  If $V$ be an $n$-dimensional linear representation of $G$ and let $G$ act component-wise on $V^m$.
  Then
  \begin{eqnarray}
    \{[V^m/G]\} &=&\Lclass^{nm}\cl \label{eq-V/G};\\
    %\{[\mathbb{P}(V)/G]\}&=&\left(\frac{\Lclass^{n}-1}{\Lclass -1}\right)\cl \label{eq-P(V)/G}
    \{[\mathbb{P}(V)/G]\}&=&\left(1+\Lclass^{1}+\dots+\Lclass^{n-1}\right)\cl \label{eq-P(V)/G}.
  \end{eqnarray}
\end{lem}
\begin{proof}
  From the vector bundle $[\nicefrac{V}{G}]\rightarrow \B{G}$ and from the second property in Definition \ref{defi-grot-stack}, one has that $\{[\nicefrac{V}{G}]\}=\Lclass^{n}\cl$. Similarly, one proves the first equation.
  
  \noindent
  Let $O$ be the origin of $V$.
  The natural map $[\nicefrac{V\setminus \{O\}}{G}]\rightarrow [\mathbb{P}(V)/G]$ is a $\mathbb{G}_m$-torsor and this implies $\{[\nicefrac{V\setminus \{O\}}{G}]\}=(\Lclass-1)\{[\mathbb{P}(V)/G]\}$.
  Moreover, $\{[\nicefrac{V\setminus \{O\}}{G}]\}=(\Lclass^{n}-1)\cl=(\Lclass-1)\{[\mathbb{P}(V)/G]\}$.
\end{proof}
%
%We stress that we do not require that $V$ is irreducible or faithful.
%
\noindent
Formula (\ref{eq-V/G}) expresses how $\cl$ is connected with $\{[\nicefrac{V^m}{G}]\}$. 
The next proposition links $\cl$ to $\{\nicefrac{V^m}{G}\}$. 
Behind this result there is the study of the difference between $\{[\nicefrac{V^m}{G}]\}$ and $\{\nicefrac{V^m}{G}\}$ in $K_0(Var_{\field})[\mathbb{L}^{-1}]$.

We write an element of $V^m$ as
\[
        v=(v_1, \dots, v_n, v_{n+1}, \dots, v_{2n}, \dots,v_{(k-1)n+1}, \dots, v_{kn},v_{kn+1},\dots, v_{m})
\]
with $k=\lfloor \nicefrac{m}{n} \rfloor$.
In other words, we consider $v\in V^m$ as a sequence of sets made by $n$ vectors each. 
Let $U$ be the subset of $V^m$ such that at least one of the sets $\{v_{jn+1}, \dots, v_{jn+n}\}$ is a basis for $V$.%, for $j=0,\dots, k$.

\noindent
We denote by $M$ the complement of $U$ in $V^m$. 
This is a closed subset of $V$, because it is defined by $k$ equations $\det(v_{jn+1}, \dots, v_{jn+n})=0$, for $j=0,\dots, k-1$. 
Therefore, $\operatorname{codim}(M)=\operatorname{codim}(\nicefrac{M}{G})=k$.
% and
% \[
%         \lim_{m\rightarrow \infty}\operatorname{codim}(M)=\lim_{m\rightarrow \infty}\operatorname{codim}(\nicefrac{M}{G})=+\infty.
% \]
We also observe that $U$ is $\Gl{n}{\field}$-invariant, because any linear transformation in $\Gl{n}{\field}$ moves a basis of $V$ into another one. 
Moreover, $\Gl{n}{\field}$ (and so $G$) acts freely on it, hence $[\nicefrac{U}{G}]=\nicefrac{U}{G}$.

\noindent
The difference $\{[\nicefrac{V^m}{G}]\}-\{\nicefrac{V^m}{G}\}$ becomes
\begin{eqnarray}
        \{[\nicefrac{V^m}{G}]\}-\{\nicefrac{V^{m}}{G}\} &=& \left(\{\mathcal{Z}\}+\{\nicefrac{U}{G}\} \right)- \left(\{\nicefrac{M}{G}\}+\{\nicefrac{U}{G}\} \right) \nonumber\\
                                                &=& \{\mathcal{Z}\}-\{\nicefrac{M}{G}\}, \label{eq-difference-VG}
\end{eqnarray}
where $\mathcal{Z}$ is a stack, complement of $\nicefrac{U}{G}$ in $[\nicefrac{V^m}{G}]$.
Similarly to $\nicefrac{M}{G}$, $\mathcal{Z}$ has codimension $k$ because both are the complement of the same object $\nicefrac{U}{G}$, but in two different environments $\nicefrac{V^{m}}{G}$ and $[\nicefrac{V^m}{G}]$ with the same dimension.
The class of the difference $\{[\nicefrac{V^m}{G}]\}-\{\nicefrac{V^{m}}{G}\}$ is so determined by the class of these complements.
%
% We see, in the next proposition, how this implies that 
% \[
%         \cl=\lim_{m\rightarrow \infty} \frac{\{\nicefrac{V^{m}}{G}\}}{\Lclass^{mn}} \in \grotcom{\Var{\field}}.
% \]

\begin{pro}[Proposition 3.1 in \cite{EkedahlStack}]\label{prop-property-BG}      
  If $V$ is an $n$-dimensional faithful linear representation of $G$, then
\begin{description}
  \item [a)] $\{\B{G}\}=\nicefrac{\{\nicefrac{GL(V)}{G}\}}{\{GL(V)\}}$;
  \item [b)] The image of $\cl$ in $\grotcom{\Var{\field}}$ is equal to $\lim_{m\rightarrow \infty} \{\nicefrac{V^{m}}{G}\} \Lclass^{-mn}$.
  %\item [c)] $\chi_c(\{\B{G}\})=1$.
\end{description}
\end{pro}
\begin{proof}
        The general linear group is a special group and we apply Lemma \ref{lem-group-basic}.\textbf{b} for $G\subseteq GL(V)$: $\{\B{G}\}=\{\nicefrac{GL(V)}{G}\}\{\B{GL(V)}\}$.
        Using Lemma \ref{lem-group-basic}.\textbf{a}, one gets $\{\B{GL(V)}\}=\nicefrac{1}{\{GL(V)\}}$ and so, we prove the first point.

        %Regarding the second one, we want to prove that 
        %\[
        %       \cl -\{\nicefrac{V^{m}}{G}\}\Lclass^{-mn}\in Fil^{-\lfloor \nicefrac{m}{n} \rfloor}(K_0(Var_{\field})[\mathbb{L}^{-1}]).
        %\]
        Using formula (\ref{eq-V/G}) and formula (\ref{eq-difference-VG}) one has 
        \begin{eqnarray*}
	  \cl -\{\nicefrac{V^{m}}{G}\}\Lclass^{-mn}
	  &=& \left(\{[\nicefrac{V^m}{G}]\}-\{\nicefrac{V^{m}}{G}\}\right)\Lclass^{-mn}\\
	  &=& \left(\{\mathcal{Z}\}-\{\nicefrac{M}{G}\}\right)\Lclass^{-mn}
        \end{eqnarray*}
        where $\mathcal{Z}$ and $\nicefrac{M}{G}$ are respectively the complement of $\nicefrac{U}{G}$ firstly seen inside of $[\nicefrac{V^m}{G}]$ and then inside of $\nicefrac{V^m}{G}$. 
        The open set $U$ was defined just before this lemma.
	
	\noindent
        Remark that $Fil^{j}(\grot{\Var{\field}}[\mathbb{L}^{-1}])=\{\nicefrac{\{X\}}{\mathbb{L}^i}: \dim{X}-i\leq j\}$. 
        Then, 
        $$\{\nicefrac{M}{G}\}\Lclass^{-mn}\in Fil^{j}(\grot{\Var{\field}}[\mathbb{L}^{-1}])\Leftrightarrow \dim{\nicefrac{M}{G}}-mn\leq j.$$
        One knows that $\dim{\nicefrac{M}{G}}-mn=-\operatorname{codim}(\nicefrac{M}{G})=-k$. 
        Thus $\{\nicefrac{M}{G}\}\Lclass^{-mn}$ belongs to $Fil^{j}$ for any $j\geq -k=\lfloor \nicefrac{m}{n} \rfloor$.
        Therefore, 
        $$\lim_{m\rightarrow \infty} \{\nicefrac{M}{G}\} \Lclass^{-mn}=0$$ 
        and, with a similar argument, $\lim_{m\rightarrow \infty} \{\mathcal{Z}\} \Lclass^{-mn}=0$. Thus, $\cl -\{\nicefrac{V^{m}}{G}\}\Lclass^{-mn}$ converges to zero in $\grotcom{\Var{\field}}$.
        %
	%We refer to Proposition 3.1 in \cite{EkedahlStack} for \textbf{c)}.
\end{proof}

% From now on, we need $\operatorname{char}(\field)=0$ to use the Bittner result \cite{Bittner2004}.
%
%The reader can fix $\field=\com$

\section{The Ekedahl invariants for finite groups}\label{sec-ekedahl-invariants}
We want to define certain cohomological maps $\operatorname{H}^k$ for $\grotcom{\Var{\field}}$. 
These (and the invariants we are going to define) need a more refined target: Let $\Lo{\Ab}$ be the group generated by the isomorphism classes $\{G\}$ of finitely generated abelian groups $G$ under the relation $\{A\oplus B\}=\{A\}+\{B\}$. 
We equip $\Lo{\Ab}$ with the discrete topology.
%
%Moreover, let $\Lof{\Ab}$ be the subgroup of $\Lo{\Ab}$ generated by the classes of finite groups.

\noindent
For clarification, $\{\mathbb{Z}\}$ and $\{\mathbb{Z}/p^n\}$ belong to $\Lo{\Ab}$ and there are elements in $\Lo{\Ab}$ that do not correspond to any group: while $\{\mathbb{Z}\}+\{\mathbb{Z}/5\}$ is the class of $\{\mathbb{Z}\oplus\mathbb{Z}/5\}$, the element $\{\mathbb{Z}\}-\{\mathbb{Z}/5\}$ is not the class of any group.

% \begin{defi}\label{def-ek-ab}
% We denote by $\Lo{\Ab}$ the group generated by isomorphism classes of finitely generated abelian groups under the relation that $\{A \oplus B\}=\{A\}+\{B\}$. 
% %
% We equip $\Lo{\Ab}$ with the discrete topology.
% 
% \noindent
% Moreover, let $\Lof{\Ab}$ be the subgroup of $\Lo{\Ab}$ generated by the classes of finite groups.
% %
% %We also denote by $\Lo{\Abgr}=\Lo{\Ab}[t,t^{-1}]$, that is the graded version of the $\Lo{\Ab}$.
% \end{defi}

If $\field=\com$, it is natural to define a cohomological map
\[
  \operatorname{H}^k: \grot{\Var{\field}}\rightarrow \Lo{\Ab},
\]
by assigning to every smooth and proper $\field$-variety $X$ the class of its integral cohomology group $\Hom{k}{X}{\inte}$.
If instead $\field$ is different from $\com$, then we send $\{X\}$ to the class $\{\Hom{k}{X}{\inte}\}$ defined as $\operatorname{dim}\Hom{k}{X}{\inte}\{\inte\}+\sum_p\{\operatorname{tor}\Hom{k}{X}{\mathbb{Z}_p}\}$.

% \footnote{%
%   In Proposition 3.4.\textbf{iii)} of \cite{EkedahlStack}, this map is called $\operatorname{NS}^k$. 
%   %
%   For any $X$ smooth and proper $\operatorname{NS}^k(\{X\})=\{\operatorname{NS}^k(X)\}$, where $\operatorname{NS}^k(X)(\bar{\field})$ is the group of algebraic codimension $k$-cycles on $X_{\bar{\field}}$ modulo homological equivalence with the natural $\operatorname{Gal}(\bar{\field}/\field)$-action. 
% }.
%
Next theorem shows that this map is well defined
% \footnote{%
%   We wrote it down in the introduction.
% } 
and that it can be extended to $\grotcom{\Var{\field}}$ sending $\nicefrac{\{X\}}{\mathbb{L}^m}$ to $\{\Hom{k+2m}{X}{\inte}\}$ for any smooth and proper variety $X$,
\[
        \operatorname{H}^k: \grotcom{\Var{\field}}\rightarrow \Lo{\Ab}.
\]
This theorem can be also proved in a slightly general setting using the corollaries in Section 7 and 9 of \cite{Manin-motifs}.

\begin{teo}\label{thm-cohomological-map-stack}
The following cohomological map
  \begin{eqnarray*}
  	\operatorname{H}^*:\grotcom{\Var{\field}}& \rightarrow & \Lo{\Ab}((t))\\
        \{Y\}           	& \mapsto     & \sum_{k\in \inte} \Homred{k}{\{Y\}} t^k.
  \end{eqnarray*}
  is well defined.
  For each $k\in \inte$, $\operatorname{H}^k:\grotcom{\Var{\field}}\rightarrow \Lo{\Ab}$ is also a continuous group homomorphism.
\end{teo}
\begin{proof}
  The proof is given by Ekedahl in \cite{EkedahlStack} via Proposition 3.2.\textbf{i)}, \textbf{ii)} and Proposition 3.3.\textbf{ii)}.
  We give an alternative proof.

  We first prove that the map $\operatorname{H}^k:\grot{\Var{\field}} \rightarrow  \Lo{\Ab}$ is well defined.
  We know that $\grot{\Var{\field}}$ is generated by the class of smooth and proper varieties modulo the relations $\{X\}+\{E\}=\{\tilde{X}\}+\{Y\}$ with $\tilde{X}$ being the blow up of $X$ along $Y$ (smooth subvariety of codimension $d$) with exceptional divisor $E$ (note that $E$ is also smooth because it is a projective bundle over $Y$, $r=\pi|_{E}:E\rightarrow Y$):
  \begin{diagram}
    \tilde{X}       & \rTo^{\pi}  & X\\
      \uInto_{h}    &             & \uInto^{j}\\
	E           &   \rTo^{r}    	 & Y.
  \end{diagram}
  Moreover, by the Leray-Hirsch Theorem (see for instance \cite{HatcherTopology}),
  \[
    \Homred{k}{E}\cong \Homred{k}{Y}\oplus \Homred{k-2}{Y}\oplus \dots \oplus \Homred{k-2(d-1)}{Y}.
  \]
  
  We want to show that $\{\operatorname{H}^k(\tilde{X})\}+\{\Homred{k}{Y}\}=\{\Homred{k}{E}\}+\{\Homred{k}{X}\}$ and therefore it is enough to show that $\operatorname{H}^k(\tilde{X})\cong\Homred{k}{X}\oplus \Homred{k-2}{Y}\oplus \dots \oplus \Homred{k-2(d-1)}{Y}$.
  
  Firstly we observe that the pushforward of the fundamental class of $\tilde{X}$ is the fundamental class of $X$, $\pi_*[\tilde{X}]=[X]$ (see \cite{FultonIntersection}).
  Now, $1$ is the dual of $[\tilde{X}]$ and, respectively, of $[X]$, $\pi_* 1 = 1$.
  Using this and the projection formula one gets that for every $y$ in $\Homred{k}{X}$ $\pi_*(1\cdot \pi^* y)=\pi_*(1)\cdot y$, that is $\pi_*\pi^* y= y$ and so $\pi_*\pi^*=\operatorname{id}_{\Homred{k}{X}}$.
  Therefore, $\pi_*:\operatorname{H}^k(\tilde{X}) \rightarrow \Homred{k}{X}$ is surjective and one constructs the isomorphism $\operatorname{H}^k(\tilde{X})=\Homred{k}{X}\oplus \operatorname{ker}(\pi_*)$ sending $x$ in $\operatorname{H}^k(\tilde{X})$ into $(\pi_*x,x-\pi^*\pi_*x)$.
  
  Calling $U=X\setminus Y$ we also have the following commutative diagram:
  \begin{diagram}
    \dots & \rTo  & \operatorname{H}^{k-1}(\tilde{X}) & \rTo  & \Homred{k-1}{U} & \rTo  & \Homred{k-2}{E} & \rTo^{h_*}  & \operatorname{H}^k(\tilde{X}) & \rTo  & \Homred{k}{U} & \rTo  &\dots\\
     & & \dTo_{\pi_*} & & \dTo_{\operatorname{id}} & & \dTo_{r_*} & & \dTo_{\pi_*} & & \dTo_{\operatorname{id}} &  &\\
    \dots & \rTo  & \Homred{k-1}{X} & \rTo  & \Homred{k-1}{U} & \rTo  & \Homred{k-2}{Y} & \rTo  & \Homred{k}{X} & \rTo^{j_*}  & \Homred{k}{U} & \rTo  &\dots
  \end{diagram}
 
  Firstly we observe that $h_*:\operatorname{ker}(r_*)\rightarrow \operatorname{ker}(\pi_*)$ is an isomorphism. 
  Indeed let $x$ be in $\operatorname{ker}(\pi_*)$. 
  Since $\pi_*x=0$, then $j_*\pi_*x=0$, but the diagram commutes and $x$ is also the kernel of $\operatorname{H}^k(\tilde{X}) \rightarrow \Homred{k}{U}$ and hence, there exists $\alpha$ in $\Homred{k-2}{E}$ mapping to $x$. 
  It is easy to see that $\alpha$ belongs to $\operatorname{ker}(r_*)$. 
  Thus the map is surjective.
  
  It is also injective because if $h_*x=0$ then there exist $\beta$ in $\Homred{k-1}{U}$ mapping to $x$. 
  The diagrams commutes and so in the second lines, $\beta$ maps to zero and, hence, there exists $z$ in $\Homred{k-1}{X}$ mapping to $\beta$. 
  The map $\pi_*$ is surjective and so there exists $z'$ in $\operatorname{H}^{k-1}(\tilde{X})$ mapping to $\beta$ in the first rows. 
  Thus $\beta$ has to map to zero and so $x=0$.
  Finally we observe that $\operatorname{ker}(r_*)$ is exactly $\Homred{k-2}{Y}\oplus \dots \oplus \Homred{k-2(d-1)}{Y}$.
  This shows that $\operatorname{H}^k:\grot{\Var{\field}} \rightarrow  \Lo{\Ab}$ is well defined.
  
  Finally, one extends this map, first to the Motivic ring and then to $\grotcom{\Var{\field}}$.
\end{proof}

% Without confusion we denote by $1$, the class of a point $\{*\}$ in $\grotcom{\Var{\field}}$ and we also denote by $1=\{\inte\}\in \Lo{\Ab}$. 
% %
% With this notation, $\Homred{*}{1}=1$.
% %
% \noindent
% We finally state the main definition.
%
\begin{defi}\label{defi-ekedahl-invariants-stacky}
%   Let $G$ be a finite group. 
%   %
%   Let $\field$ be a field.
%   %
%   Let $\{\B{G}\}$ be the class of the classifying stack of $G$ in $\grot{\Stck{\field}}$.
  %
  The $i$-th Ekedahl invariant of a finite group $G$ is 
  \[
    \eke{i}{G}=\Homred{-i}{\{\B{G}\}}\in \Lo{\Ab}.
  \]
  We say that $\eke{i}{G}$ are \emph{trivial} if $\eke{i}{G}=0$ for $i\neq 0$.
\end{defi}

\section{A non-stacky definition}\label{sec-nonstaky}
In this section we present an equivalent definition for these invariants that does not involve the concept of algebraic stacks.

Let $V$ be a $n$-dimensional faithful $\field$-representation of a finite group $G$.
The group $G$ acts component-wise on $V^m$.
Consider the quotient scheme $\nicefrac{V^m}{G}$ that is usually a singular scheme. 
\begin{pro}\label{pro-coh-stabilizes}
  For $m$ large enough, the \emph{cohomology} $\Homred{-k}{\{\nicefrac{V^{m}}{G}\} \Lclass^{-mn}}$ stabilizes.
\end{pro}
\begin{proof}
  We have seen in Proposition \ref{prop-property-BG}.\textbf{b)} that 
  $$\cl=\lim_{m\rightarrow \infty} \{\nicefrac{V^{m}}{G}\} \Lclass^{-mn}\in \grotcom{\Var{\field}}.$$
  From Theorem \ref{thm-cohomological-map-stack}, the map $\operatorname{H}^k$ is continuous and the topology of $\Lo{\Ab}$ is discrete, so for \emph{$m$ large enough}, $\Homred{-k}{\{\B{G}\}}=\Homred{-k}{\{\nicefrac{V^{m}}{G}\} \Lclass^{-mn}}$.
%   \[
%     \Homred{-i}{\{\B{G}\}}=\Hom{-i}{\{\nicefrac{V^{m}}{G}\} \Lclass^{-mn}}{\inte}.
%   \]
\end{proof}
\noindent
We set $m(i,V)$ to be the positive integer where this cohomology stabilizes.

Let $X$ be a smooth and proper resolution of $\nicefrac{V^m}{G}$:
\[
  X\xrightarrow{\,\,\,\,\,\,\,\pi\,\,\,\,\,\,\,} \nicefrac{V^m}{G}.
\]

\begin{defi}\label{defi-ekedahl-invariants-non-stacky}
  Let $V$, $G$, $m(i,V)$ and $X$ defined as before.
  The $i$-th Ekedahl invariant is defined as follows:
  \[
    \eke{i}{G}=\{\Hom{2m(i,V)-i}{X}{\inte}\}+\sum_j n_j\{\Hom{2m(i,V)-i}{X_j}{\inte}\} \in \Lo{\Ab},
  \]
  where $\{\nicefrac{V^{m(i,V)}}{G}\}\in \grot{\Var{\field}}$ is written as the sum of classes of smooth and proper varieties $\{X\}$ and $\{X_j\}$, $\{\nicefrac{V^{m(i,V)}}{G}\}=\{X\}+\sum_j n_j\{X_j\}$. 
\end{defi}
\noindent
This definition does not involve the theory of algebraic stacks, but furthermore it gives a concrete way to compute these invariants.

Now, we show that the two given definitions of Ekedahl invariants are equivalent.
Moreover, by proving this, we prove that the new definition is independent by the choice of the faithful representation $V$ and of the smooth and proper resolution $X$.

\begin{pro}\label{pro-defi-equivalent}
  Definition \ref{defi-ekedahl-invariants-non-stacky} is equivalent to Definition \ref{defi-ekedahl-invariants-stacky}.
\end{pro}
\begin{proof}
  For $m=m(i,V)$, 
  \[
    \Homred{-i}{\{\B{G}\}}=\Hom{-i}{\{\nicefrac{V^{m}}{G}\} \Lclass^{-mn}}{\inte}=\Hom{2mn-i}{\{\nicefrac{V^{m}}{G}\}}{\inte},
  \]
  %$\Homred{-i}{\{\B{G}\}}=\{\Hom{2mn-i}{\{\nicefrac{V^{m}}{G}\}}{\inte}\}$, 
  where the shifting $2mn$ comes from the multiplication for $\Lclass^{-mn}$.

  By definition $X$ is a proper resolution of the singularities of $\nicefrac{V^{m}}{G}$.
  For the Bittner results showed in Section \ref{sec-Preliminaries}, we write $\{\nicefrac{V^{m}}{G}\}$ as a suitable sum $\{X\}+\sum_j n_j\{X_j\}$ where $\{X\}$ is smooth, proper and birational to $\nicefrac{V^{m}}{G}$, the $X_j$'s are smooth and proper with dimension strictly less then $\operatorname{dim}(\nicefrac{V^{m}}{G})=mn$ and $n_j\in \inte$.
  Therefore,
  \[
    \eke{i}{G}=\Homred{-i}{\{\B{G}\}}=\{\Hom{2mn-i}{X}{\inte}\}+\sum_j n_j\{\Hom{2mn-i}{X_j}{\inte}\}.
  \]
\end{proof}

\subsection{The state of the art}\label{sec-literature}
The following theorem links to the Noether problem.

\begin{teo}[Theorem 5.1 in \cite{Ekedahl-inv}]\label{thm-ekedahl}
  We denote by $B_0(G)^{\vee}$ the pontryagin dual of the Bogomolov multiplier of the group $G$.
  If $G$ is a finite group, %and if $\field$ is an algebraically closed field of characteristic zero, 
  then
  \begin{description}
    \item[a)] $\eke{i}{G}=0$, for $i<0$;
    \item[b)] $\eke{0}{G}=\{\mathbb{Z}\}$;
    \item[c)] $\eke{1}{G}=0$;
    \item[d)] $\eke{2}{G}=\{B_0(G)^{\vee}\}+\alpha\{\inte\}$ for some integer $\alpha$.
%     \item[d)] $\eke{2}{G}=\{B_0(G)^{\vee}\}$, where $B_0(G)^{\vee}$ is the dual of the Bogomolov multiplier of the group $G$;
%     \item[e)] For $i>0$, $\eke{i}{G}$ is the sum (with signs) of classes of finite groups in $\Lo{\Ab}$.
  \end{description}
\end{teo}
\begin{proof}
  By Definition \ref{defi-ekedahl-invariants-non-stacky} and setting for simplicity $m=m(i,V)$, 
  \[
    \eke{i}{G}=\{\Hom{2mn-i}{X}{\inte}\}+\sum_j n_j\{\Hom{2mn-i}{X_j}{\inte}\},
  \]
  where $X$ is a smooth and proper resolution of $\nicefrac{V^m}{G}$; $V$ is a $n$-dimensional faithful $\field$-representation of a finite group $G$ and $\{\nicefrac{V^m}{G}\}$ is written in $\grotcom{\Var{\field}}$ as the sum of classes of smooth and proper varieties $\{X_j\}$, $\{\nicefrac{V^m}{G}\}=\{X\}+\sum_j n_j\{X_j\}$. 
  
  Let $i=0$. The only surviving cohomology is $\Hom{2mn}{X}{\inte}=\inte$, because $\operatorname{dim}(X_j)<\operatorname{dim}(\nicefrac{V^{m}}{G})=mn$. Thus, $\eke{0}{G}=\{\Hom{2mn}{X}{\inte}\}=\{\inte\}$.
        
  If $i=1$, for similar reasons, $\eke{1}{G}=\{\Hom{2mn-1}{X}{\inte}\}$.
  Since $X$ is birational to $\nicefrac{V^{m}}{G}$, one has the inclusion $\field(X)\simeq \field(\nicefrac{V}{G})\subseteq \field(V)$ and hence $X$ is unirational and, therefore, simply connected. 
  Thus, using the result of Serre in \cite{Serre1959}, $\Hom{2mn-1}{X}{\inte}\simeq\operatorname{H}_1(X; \inte)=0$ and thus $\eke{1}{G}=0$.
  
  Regarding $\eke{2}{G}$, one firstly observes, by Poincar\'e duality, that
  \[
    \operatorname{tor}(\Hom{2mn-2}{X}{\inte})\cong \operatorname{tor}(\operatorname{H}_3(X; \inte)).
  \]
  Artin and Mumford  have proved in \cite{ArtinMumford1972} that $\operatorname{tor}(\Hom{3}{X}{\inte})$ is a birational invariant and Bogomolov in \cite{Bogomolov1988} proved that this is exactly $B_0(G)$.
  Therefore, we have proved that $\eke{2}{G}=\{B_0(G)^{\vee}\}+\alpha\{\inte\}$ for some integer $\alpha$.
%   Up to now, we have proved that $\eke{2}{G}=\{B_0(G)^{\vee}\}+\alpha\{\inte\}$ for an integer $\alpha$. 
%   %
%   The point \textbf{e)} of the statement also implies that $\alpha=0$.
%   %
%   We refer to Theorem 5.1 in \cite{Ekedahl-inv} for the point \textbf{e)}.
\end{proof}

  \begin{Obs}
    Bogomolov proved in Theorem 1.1 of \cite{Bogomolov1988} that if $X$ is smooth, proper and unirational the Brauer group $Br_v(\mathbb{K})$ is isomorphic to $\operatorname{tor}(\Hom{3}{X}{\inte})$, with $\mathbb{K}=\field(X)$.
    Moreover he defined $Br_v(G)=Br_v(\field(X))$, where $X$ is smooth, proper and birational to $\nicefrac{V}{G}$ with $V$ being any generically free representation of $G$. % and so $Br_v(G)=Br_v(\nicefrac{V^m}{G})$.
    Thus, in Theorem 3.1 of \cite{Bogomolov1988}, he has proved that $Br_v(G)=B_0(G)$.
  \end{Obs}

\noindent
Using these results, one finds the first examples of group with non trivial Ekedahl invariants.
\begin{cor}[Non triviality]\label{prop-NP}
  The second Ekedahl invariant is non trivial for every algebraically closed field $\field$ with $\operatorname{char}(\field)=0$ and for the groups of order $p^9$ given by Saltman in \cite{Saltman1984}, of order $p^6$ given by Bogomolov in \cite{Bogomolov1988} and the group of order $p^5$ in the the isoclinism family $\phi_{10}$ (see \cite{HKM-noether}).
  Moreover in these cases, $\{BG\}\neq \{*\}$ in $\grotcom{\Var{\field}}$.
\end{cor}
\begin{proof}
  The Bogomolov multiplier is always a finite group and so if $B_0(G)\neq 0$, then
  $\eke{2}{G}=\{B_0(G)^{\vee}\}+\alpha\{\inte\}\neq 0$.
\end{proof}

\noindent
For this reasons it is natural to ask the following questions:
\begin{Question} 
Does, for $i>2$, $\eke{i}{G}\neq 0$ imply a negative answer to the Noether problem? In other words, are the Ekedahl invariants obstructions to the rationality of the extension $\field(V)^G/\field$?
\end{Question}

Another connection between the Noether problem and the non-triviality of $\cl$ is also the next proposition.
\begin{pro}[Corollary 5.8 in \cite{Ekedahl-inv}]
  %For $\field=\mathbb{Q}$, 
  $\{\B{\nicefrac{\inte}{47\inte}}\}\neq \{*\}$ in $\grotcom{\Var{\rat}}$.
\end{pro}

\noindent
This leads to another natural question:

\begin{Question} 
Is $\cl \neq \{*\}$ an obstruction to the rationality of the extension $\field(V)^G/\field$?
\end{Question}

Regarding item \textbf{d)} of the previous theorem, Ekedahl actually proved a more precise statement.

\begin{teo}[Theorem 5.1 of \cite{Ekedahl-inv}]
  For $i>0$, $\eke{i}{G}$ is the sum (with signs) of classes of finite groups in $\Lo{\Ab}$.
\end{teo}
\begin{proof}
  We refer to point \textbf{e)} of Theorem 5.1 in \cite{Ekedahl-inv}.
\end{proof}

\begin{cor}
  The second Ekedahl invariant is exactly $\eke{2}{G}=\{B_0(G)^{\vee}\}$.
\end{cor}
\begin{proof}
  We already proved in Theorem \ref{thm-ekedahl}.\textbf{d)} that $\eke{2}{G}=\{B_0(G)^{\vee}\}+\alpha\{\inte\}$ for some integer $\alpha$.
  Using the previous theorem one gets $\alpha=0$.
\end{proof}

To the author’s knowledge, there are no examples in literature of finite group $G$ such that $B_0(G)=0$ and $\eke{3}{G}\neq 0$. 
Vice versa a lot of groups have trivial Ekedahl invariants.
%
% One way to study this is focusing on the class $\cl$.
% %
% Indeed if $\cl=1$ then the $\eke{i}{G}$ are trivial.
%
\begin{teo}[State of art for finite groups]\label{pro-BG=1}
  Assume one of the following cases: 
  \begin{description}
    \item[1)] if $G$ is the symmetric group and for every field $\field$;
    \item[2)] if $G \subset GL_1$ and for every field $\field$ (in particular, if $G$ is a cyclic group); 
    \item[3)] if $G$ is a unipotent finite group and for every field $\field$;
    \item[4)] if $G$ is a finite subgroup of the group of affine transformations of $\mathbb{A}^1_{\field}$ and for every algebraically closed field $\field$;
    \item[5)] if $G \subset GL_3(\com)$ and for $\field=\com$.
  \end{description}
  
  \noindent
  Then $\{\B{G}\}=\{*\}\in \grotcom{\Var{\field}}$ and the trivial Ekedahl invariants are trivial.
\end{teo}
\begin{proof}
  See in order Theorem 4.3, Proposition 3.2, Corollary 3.9, Example \textit{ii)} at page 8 in \cite{Ekedahl-inv}, and Theorem 2.5 in \cite{Martino-TheEkedahl}.
\end{proof}

\noindent
Recently the author has proved also the following facts:
\begin{teo}[Thm 3.1 in \cite{Martino-TheEkedahl}]
  Let $G$ be a finite subgroup of $\Gl{n}{\com}$ and let $H$ be the image of $G$ under the canonical projection $\Gl{n}{\com}\rightarrow \operatorname{PGL}_n(\com)$.
  
  \noindent
  If $H$ is abelian and if the quotient $\nicefrac{\mathbb{P}^{n-1}_{\com}}{H}$ has only zero dimensional singularities, then for every integer $k$
  \[
    \eke{k}{G}+\eke{k+2}{G}+\dots+\eke{k+2(n-1)}{G}=\{\Hom{-k}{X}{\inte}\},
    %\Homred{k}{\{\B{G}\}(1+\Lclass+\dots+\Lclass^{n-1})}=\Homred{k}{\{\nicefrac{\mathbb{P}^{n-1}}{H}\}}.
  \]
  where $X$ is a smooth and proper resolution of $\nicefrac{\mathbb{P}^{n-1}_{\com}}{H}$.
  %Moreover $\{\Hom{2j}{X}{\inte}\}={\inte}$ for $0\leq j\leq n-1$.
\end{teo}

\begin{teo}[Thm 4.4 in \cite{Martino-TheEkedahl}]
  The Ekedahl invariants of the fifth discrete Heisenberg group, $\eke{i}{\hei{5}}$, are trivial.
\end{teo}

We observe that, in \cite{Martino-TheEkedahl}, the author approaches the the study of the Ekedahl invariants of $\hei{p}$ for a general $p$, but the assumption $p=5$ is necessary because of the difficulties to extend the technical result in Theorem 4.7 of \cite{Martino-TheEkedahl}.

\noindent
Moreover, in Theorem 1.9 of \cite{Kang-Rationality}, Kang has proved that the Noether problem for the Heisenberg group $\hei{p}$ has positive answer.
Then, it is natural to conjecture that:
\begin{conj}
  The Ekedahl invariants of the $p$-discrete Heisenberg group $\hei{p}$ are trivial.
\end{conj}

\textbf{Acknowledgements}\\
I thank Angelo Vistoli for the great mathematical support in this subject.

%============================================================
%-----------------------------
%\newpage
%\clearpage
\begin{scriptsize}
\addcontentsline{toc}{section}{Bibliography}
\bibliographystyle{siam}
\bibliography{IntroToeG}
\end{scriptsize}

\noindent
 {\scshape Ivan Martino}\\
 {\scshape Department of Mathematics, University of Fribourg,\\ 1700 Fribourg, Switzerland}.\\
 {\itshape E-mail address}: \texttt{ivan.martino@unifr.ch}
 
% \noindent
%  {\scshape Department of Mathematics, Stockholm University, SE-10691 Stockholm, Sweden}.\\
%  {\itshape E-mail address}: \texttt{martino@math.su.se}
\end{document}